\documentclass[12pt]{amsart}

\usepackage[letterpaper,margin=.6in]{geometry}
\usepackage[colorlinks,linkcolor=blue,citecolor=black!75!red]{hyperref}

\usepackage{url}
\usepackage{color}
\usepackage[all]{xypic}

\usepackage{latexsym}
\usepackage{amssymb}
\usepackage{amsfonts}
\usepackage{amscd}
\usepackage{amsmath,amsthm}
\usepackage{verbatim}

\usepackage{tikz}
\usetikzlibrary{matrix,arrows,decorations.pathmorphing,fit,cd}
\tikzset{myboxgroup/.style={draw, densely dotted}} 

\usepackage{cleveref}

\newtheorem{lemma}{Lemma}[section]
\newtheorem{proposition}[lemma]{Proposition}
\newtheorem{theorem}[lemma]{Theorem}

{

}

\theoremstyle{definition}

\theoremstyle{remark}

\makeatletter
\let\xx@thm\@thm
\AtBeginDocument{\let\@thm\xx@thm}
\makeatother

\crefname{section}{Section}{Sections}
\crefformat{section}{#2section~#1#3}
\Crefformat{section}{#2Section~#1#3}

\crefformat{subsection}{#2\S#1#3}
\Crefformat{subsection}{#2\S#1#3}
\crefrangeformat{subsection}{\S\S#3#1#4--#5#2#6}
\Crefrangeformat{subsection}{\S\S#3#1#4--#5#2#6}
\crefmultiformat{subsection}{\S\S#2#1#3}{ and~#2#1#3}{, #2#1#3}{ and~#2#1#3}
\Crefmultiformat{subsection}{\S\S#2#1#3}{ and~#2#1#3}{, #2#1#3}{ and~#2#1#3}
\crefrangemultiformat{subsection}{\S\S#3#1#4--#5#2#6}{ and~#3#1#4--#5#2#6}{, #3#1#4--#5#2#6}{ and~#3#1#4--#5#2#6}
\Crefrangemultiformat{subsection}{\S\S#3#1#4--#5#2#6}{ and~#3#1#4--#5#2#6}{, #3#1#4--#5#2#6}{ and~#3#1#4--#5#2#6}

\crefname{definition}{Definition}{Definitions}
\crefformat{definition}{#2Definition~#1#3}
\Crefformat{definition}{#2Definition~#1#3}

\crefname{definitionnodiamond}{Definition}{Definitions}
\crefformat{definitionnodiamond}{#2Definition~#1#3}
\Crefformat{definitionnodiamond}{#2Definition~#1#3}

\crefname{example}{Example}{Examples}
\crefformat{example}{#2Example~#1#3}
\Crefformat{example}{#2Example~#1#3}

\crefname{examplenodiamond}{Example}{Examples}
\crefformat{examplenodiamond}{#2Example~#1#3}
\Crefformat{examplenodiamond}{#2Example~#1#3}

\crefname{remark}{Remark}{Remarks}
\crefformat{remark}{#2Remark~#1#3}
\Crefformat{remark}{#2Remark~#1#3}

\crefname{remarknodiamond}{Remark}{Remarks}
\crefformat{remarknodiamond}{#2Remark~#1#3}
\Crefformat{remarknodiamond}{#2Remark~#1#3}

\crefname{convention}{Convention}{Conventions}
\crefformat{convention}{#2Convention~#1#3}
\Crefformat{convention}{#2Convention~#1#3}

\crefname{notation}{Notation}{Notations}
\crefformat{notation}{#2Notation~#1#3}
\Crefformat{notation}{#2Notation~#1#3}

\crefname{notationnodiamond}{Notation}{Notations}
\crefformat{notationnodiamond}{#2Notation~#1#3}
\Crefformat{notationnodiamond}{#2Notation~#1#3}

\crefname{lemma}{Lemma}{Lemmas}
\crefformat{lemma}{#2Lemma~#1#3}
\Crefformat{lemma}{#2Lemma~#1#3}

\crefname{proposition}{Proposition}{Propositions}
\crefformat{proposition}{#2Proposition~#1#3}
\Crefformat{proposition}{#2Proposition~#1#3}

\crefname{corollary}{Corollary}{Corollaries}
\crefformat{corollary}{#2Corollary~#1#3}
\Crefformat{corollary}{#2Corollary~#1#3}

\crefname{theorem}{Theorem}{Theorems}
\crefformat{theorem}{#2Theorem~#1#3}
\Crefformat{theorem}{#2Theorem~#1#3}

\crefname{assumption}{Assumption}{Assumptions}
\crefformat{assumption}{#2Assumption~#1#3}
\Crefformat{assumption}{#2Assumption~#1#3}

\crefname{enumi}{}{}
\crefformat{enumi}{(#2#1#3)}
\Crefformat{enumi}{(#2#1#3)}

\crefname{equation}{}{}
\crefformat{equation}{(#2#1#3)}
\Crefformat{equation}{(#2#1#3)}

\crefname{align}{}{}
\crefformat{align}{(#2#1#3)}
\Crefformat{align}{(#2#1#3)}

\crefname{proofstep}{Step}{Steps}
\crefformat{proofstep}{#2Step~#1#3}
\Crefformat{proofstep}{#2Step~#1#3}

\crefname{table}{Table}{Tables}
\crefformat{table}{#2Table~#1#3}
\Crefformat{table}{#2Table~#1#3}

\numberwithin{equation}{section}

\def\CC{{\mathbb C}}

\def\PP{{\mathbb P}}
\def\QQ{{\mathbb Q}}

\def\ZZ{{\mathbb Z}}

\def\0ol{{\bar 0}}
\def\1ol{{\bar 1}}
\def\2ol{{\bar 2}}
\def\ol2{{\bar 2}}
\def\3ol{{\bar 3}}
\def\4ol{{\bar 4}}
\def\5ol{{\bar 5}}
\def\6ol{{\bar 6}}
\def\7ol{{\bar 7}}
\def\8ol{{\bar 8}}
\def\9ol{{\bar 9}}

\def\bold0{{\bf 0}}
\def\bold1{{\bf 1}}
\def\bold2{{\bf 2}} 
\def\bold3{{\bf  3}}
\def\bold4{{\bf 4}}
\def\bold5{{\bf 5}}
\def\bold6{{\bf 6}}
\def\bold7{{\bf 7}}
\def\bold8{{\bf 8}}
\def\bold9{{\bf 9}}

\def\P2Skly{\PP^2_{Skly}}

\def\a{\alpha}
\def\b{\beta}

\def\s{\sigma}

\def\ve{\varepsilon}

\def\D{\Delta}

\def\sD{{\sf D}}

\def\sft{{\sf t}}

\def\sfx{{\sf x}}
\def\sfy{{\sf y}}

\def\cal{\mathcal}

\def\cE{{\cal E}}
\def\cF{{\cal F}}

\def\cL{{\cal L}}

\def\cO{{\cal O}}

\def\Aut{\operatorname{Aut}}

\def\coh{{\sf coh}}

\def\dirlim{\mathop{\vtop{\baselineskip -100pt\lineskip -1pt\lineskiplimit 0pt
\setbox0\hbox{lim}\copy0\hbox to \wd0{\rightarrowfill}}}\limits}
\def\invlim{\mathop{\vtop{\baselineskip -100pt\lineskip -1pt\lineskiplimit 0pt
\setbox0\hbox{lim}\copy0\hbox to \wd0{\leftarrowfill}}}\limits}

\def\I11{{1 \kern -0.8pt \! \mbox{l}}}
\def\mumu{{\mu\kern-4.2pt\mu}}
\def\bfmu{{\mu\kern-4.2pt\mu}}
\def\2slash{\backslash \! \backslash}

\def\hdot{{\:\raisebox{3.2pt}{\text{\circle*{3.4}}}}}

\newcommand{\pr}{\operatorname{pr}}
\newcommand{\sfsum}{\operatorname{\mathsf{sum}}}

\makeatletter
\def\l@subsection{\@tocline{2}{0pt}{2.75pc}{5pc}{}}
\makeatother

\makeindex

\begin{document}

\title[Finite quotients of powers of an elliptic curve]{Finite quotients of powers of an elliptic curve}

\author{Alex Chirvasitu, Ryo Kanda, and S. Paul Smith}

\address[Alex Chirvasitu]{Department of Mathematics, University at
  Buffalo, Buffalo, NY 14260-2900, USA.}
\email{achirvas@buffalo.edu}

\address[Ryo Kanda]{Department of Mathematics, Graduate School of Science, Osaka City University, 3-3-138, Sugimoto, Sumiyoshi, Osaka, 558-8585, Japan.}
\email{ryo.kanda.math@gmail.com}

\address[S. Paul Smith]{Department of Mathematics, Box 354350,
  University of Washington, Seattle, WA 98195, USA.}
\email{smith@math.washington.edu}

\subjclass[2010]{14L30 (Primary), 14E20, 14H52, 14A22, 16S38 (Secondary)}

\keywords{Quotient variety; \'etale cover; symmetric group; elliptic algebra; characteristic variety}


\begin{abstract}
Let $E$ be an elliptic curve. When the symmetric group $\Sigma_{g+1}$ of order $(g+1)!$ acts on $E^{g+1}$ in the natural way, the subgroup $E_0^{g+1}$, consisting of those $(g+1)$-tuples whose coordinates sum to zero, is stable under the action of $\Sigma_{g+1}$. It is isomorphic to $E^g$. This paper concerns the structure of the quotient variety $E^g/\Sigma$ when $\Sigma$ is a subgroup of $\Sigma_{g+1}$ generated by simple transpositions. In an earlier paper we observed that $E^g/\Sigma$ is a bundle over a suitable power, $E^N$, with fibers that are products of projective spaces. This paper shows that $E^g/\Sigma$ has an \'etale cover by a product of copies of $E$ and projective spaces with an abelian Galois group.
\end{abstract}

\maketitle

\tableofcontents{}

\section{Introduction}
\label{sect.intro}
  
 Let $A$ be an abelian variety.
 The problem of describing $A/G$ when $G$ is a finite subgroup of $ \Aut(A)$ has attracted some attention.\footnote{
 $\Aut(A)$ denotes the group of automorphisms of $A$ as an algebraic group.} See, for example, the introduction to \cite{AA18}.
By \cite[Thm.~1.3]{AA18}, if $A/G$ is smooth and $\dim(A^G)=0$, then $A$ is a product of elliptic curves and $A/G$ is a 
 product of projective spaces. The authors of {\it loc. cit.} confine their attention to {\it complex } abelian varieties.
  
Let $E$ be an elliptic curve and $E^g$ the product of $g$ copies of $E$.
This paper describes the structure of certain quotient varieties $E^g/\Sigma$ when $\Sigma$ is a particular kind of 
 finite subgroup of $\Aut(E^g)$. 
 
 We work over an algebraically closed field $\Bbbk$.

 \subsection{Quotients of powers of an elliptic curve}  
 \label{sect.the.problem} 
Identify $E^g$ with the subgroup $E^{g+1}_0$ of $E^{g+1}$ consisting of those points whose coordinates  sum to $0$ via the injection $\ve\colon E^{g}\to E^{g+1}$ given by
\begin{equation*}
	\ve(z_{1},\ldots,z_{g})\;=\;(z_{1},z_{2}-z_{1},\ldots,z_{g}-z_{g-1},-z_{g}).
\end{equation*}
The natural action of the symmetric group $\Sigma_{g+1}$ of order $(g+1)!$ on 
$E^{g+1}$ sends $E^{g+1}_0$ to itself.  
Let $\Sigma$ be a subgroup of $\Sigma_{g+1}$ generated by some of the simple reflections $(i,i+1)$. 
The problem we address is this: 
\begin{equation}
\label{problem}
\text{\it what is the structure of the quotient variety $E^g/\Sigma$ when $E^g$ is identified with $E_0^{g+1}$?}
\end{equation}
To state our main result we need some notation. Let  
\begin{align*}
J & \; := \; \text{the points in $\{1,\ldots,g+1\}$ fixed by $\Sigma$},
\\
I_1,\ldots,I_s  & \; := \; \text{the $\Sigma$-orbits in $\{1,\ldots,g+1\}$ of size $\ge 2$},
\\
\Sigma_{I_\a} & \; :=\; \text{the group of permutations of $I_\a$},
\\
i_\a & \; :=\; |I_\a|,
\\
d & \; :=\; \gcd\{i_1,\ldots,i_s\}.
\end{align*}
For $m \in \ZZ$, let $m:E \to E$ be the multiplication by $m$ map and let $E[m]$ denote its kernel. 
Define
\begin{equation}\label{eq:1}
  E[i_1]\hdot \cdots\hdot E[i_s] \; := \; \ker\left(E[i_1]\times\cdots\times E[i_s] \, \longrightarrow \, E[d]\right)
\end{equation}
where the morphism on the right is the map
$$
(z_1,\ldots,z_s) \, \mapsto \, \frac{i_1}{d}z_1 + \cdots + \frac{i_s}{d}z_s.
$$
In \cite{CKS2} we showed $E^g/\Sigma$ is a bundle over $E^{|J|+s-1}$ with fibers  
isomorphic to $\PP^{i_1-1} \times  \cdots \times \PP^{i_s-1}$ (see \cref{thm1} below). The main result in this paper provides
a nice \'etale cover of $E^g/\Sigma$.

\begin{theorem}
\label{main.thm}
Assume $\mathrm{char}(\Bbbk)$ does not divide the product $i_1\cdots i_s$. 
\begin{enumerate}
  \item 
  If $J\ne \varnothing$,  there is an \'etale cover
  $$
\Psi:E^{|J|+s-1} \times \PP^{i_1-1} \times  \cdots \times \PP^{i_s-1} 
\; \longrightarrow \; E^g/\Sigma
$$
with Galois group $E[i_1]\times\cdots\times E[i_s]$. 
  \item 
  If $J = \varnothing$, there is an \'etale cover
$$
\Psi:E^{s-1} \times \PP^{i_1-1} \times  \cdots \times \PP^{i_s-1} 
\; \longrightarrow \; E^g/\Sigma
$$
with Galois group $E[i_1] \hdot \cdots \hdot E[i_s]$.
\end{enumerate}
\end{theorem}

As is well-known, $E^g/\Sigma_{g+1} \cong \PP^g$. 
Another simple case of this theorem is the fact that the symmetric power, $S^rE$, is the quotient of $E \times \PP^{r-1}$ with
respect to a free action of $E[r]$.

In \cref{sect.lift} we show that certain ``translation automorphisms'' of $E^g/\Sigma$ lift to the \'etale cover. 
That gives a positive answer to a question that arises from our work on elliptic algebras in \cite{CKS2}. 
  
\subsection{Origin of the problem}
\label{origin1}
The question in \cref{problem} came up when we were studying the elliptic algebras $Q_{n,k}(E,\tau)$ introduced by
Feigin and  Odesskii in the late 1980's \cite{FO-Kiev,FO89}. Much of the representation theory of $Q_{n,k}(E,\tau)$
is controlled by its ``characteristic variety''. The main result in \cite{CKS2} shows that the characteristic variety, which was 
defined by Feigin and Odesskii in terms of the sheaf $\cL_{n/k}$ defined in \cref{sect.Lnk} below,
is a quotient, $E^g/\Sigma$, of the form described in \cref{sect.the.problem}.

Thus, the question ``what  is the structure of  the characteristic variety'' reduces to a problem in algebraic geometry 
that is of interest independently of its connection to Feigin and  Odesskii's algebras.  

\subsubsection{}
The remarks in \cite[\S3.3]{FO89} suggest that some part of the results in this paper may have been known to 
Feigin and Odesskii in the late 1980's. 
However, they did not describe the characteristic variety as a quotient of $E^g$ or provide proofs of their remarks in \cite[\S3.3]{FO89}
so we are doing that here. The proofs were not obvious to us.

\subsection{Relation to a Fourier-Mukai transform}
\label{origin2}
Another indication of the significance of $E^g/\Sigma$ is its relation to one of the most famous Fourier-Mukai transforms.

Let $\cE(k,n)$ denote the set of isomorphism classes of indecomposable locally free $\cO_E$-modules of rank $k$ and degree $n$.
Suppose that $n>k \ge 1$ are coprime. There is a well-known auto-equivalence of the bounded derived category 
$\sD^b(\coh(E))$ that sends $\cE(1,0)$ bijectively to $\cE(k,n)$; that auto-equivalence is the Fourier-Mukai transform
associated to a bundle on $E \times E$. There is a less well-known way to obtain such an auto-equivalence: there is an integer $g \ge 2$
and an invertible $\cO_{E^g}$-module $\cL_{n/k}$ such that 
the functor
$$
 {\bf R}\pr_{1*}  ( \cL_{n/k} \otimes^{\bf L} \pr_g^* (\, \cdot\,))
$$
is an auto-equivalence of $\sD^b(\coh(E))$ that sends $\cE(1,0)$ bijectively to $\cE(k,n)$. Since $\cL_{n/k}$ is generated by its global sections and $\dim H^0(E^g,\cL_{n/k})=n$ one can ask for a description of the image of the morphism
$$
\Phi :E^g \, \longrightarrow \, \PP^{n-1} \,=\, \PP\big(H^0(E^g,\cL_{n/k})^*\big).
$$
The characteristic variety in \cref{origin1}  is, by definition, the image of $\Phi$. It is of the form $E^g/\Sigma$.

We refer the reader to \cite[\S\S1.2 and~1.3]{CKS2} for more about this.

\subsection{The integer $g$ and the sheaf $\cL_{n/k}$}
\label{sect.Lnk}
We now explain how $g$ and $\cL_{n/k}$ are determined by the pair $(n,k)$ when $n>k \ge 1$ and $\gcd\{n,k\}=1$.

Given a locally free $\cO_E$-module $\cF$, we denote by 
$$
c_1(\cF) \; :=\; \begin{pmatrix}    {\rm deg}(\cF)    \\ {\rm rank}(\cF) \end{pmatrix}
$$
its first Chern class. The action of ${\rm PSL}(2,\ZZ)$ on $\PP^1_\QQ = \QQ \sqcup \{\infty\}$ by fractional linear transformations
$$
\begin{pmatrix} a & b \\ c & d  \end{pmatrix} \cdot z \; :=\; \frac{az+b}{cz+d}
$$
is transitive. The following question arises: given $A \in {\rm PSL}(2,\ZZ)$ such that $A\cdot\frac{0}{1} = \frac{n}{k}$,
express $A$ in terms of the ``standard generators''
$$
S \; = \;  \begin{pmatrix} 0 & -1 \\ 1 & 0 \end{pmatrix}
\qquad \text{and} \qquad 
T \; = \;  \begin{pmatrix} 1 & 1 \\ 0 & 1 \end{pmatrix}.
$$
There are unique integers $g \ge 1$ and $n_1,\ldots,n_g$, all $\ge 2$, such that
$$
A \;=\; T^{n_1}ST^{n_2}S \cdots ST^{n_g}.
$$
This is the origin of the integer $g$ and the integers $n_1,\ldots,n_g$\footnote{The integers $n_1,\ldots,n_g$ can be obtained from 
$(n,k)$ via a version of the Euclidean algorithm.}
 that are used to construct $\cL_{n/k}$ according to the 
formula
$$
\cL_{n/k} \;:=\;  \big(\cL^{n_1} \boxtimes \cdots \boxtimes \cL^{n_g}\big)\Bigg(\bigotimes_{j=1}^{g-1} \pr_{j,j+1}^*(\cL \boxtimes \cL)(-\D')\Bigg)
$$
where $\cL=\cO_E((0))$, $\D'=\{(z,-z) \; | \; z \in E\} \subseteq E^2$, $\pr_j:E^g \to E$ is the projection to the $j^{\rm th}$ factor, and 
$\pr_{j,j+1} :E^g \to E^2$ is the projection to the $j^{\rm th}$ and $(j+1)^{\rm th}$ factors.

\subsection{Relation to the results in \cite{AA18}}
\label{sect.aa1}

Shortly after writing the first version of this paper we realized that \cref{main.thm}, and \cref{thm1} which we proved in 
\cite{CKS2}, have substantial overlap with \cite[Prop.~2.9]{AA18} once one makes the appropriate translation between the two papers. 
Nevertheless, our approach is a little different, perhaps more elementary, though proving less, works for an arbitrary field $\Bbbk$,
and  is a little more explicit. Our formulation of the results is also easier to apply to elliptic algebras.\footnote{Computing the 
integers $n_1,\ldots,n_g$ in \cref{sect.Lnk} is straightforward; $\Sigma$ is the subgroup of $\Sigma_{g+1}$ generated by the 
transpositions $(i,i+1)$ for which $n_i=2$.}

We now explain how to translate between this paper and \cite{AA18} (this won't make 
complete sense until we set up notation in  \cref{sect.results}). Our $E^g$ and $\Sigma$ are special cases of 
the abelian variety $A$ and the group $G$ in \cite{AA18}.  In \cite{AA18}, the \'etale cover of 
$A/G$ is obtained by first describing a $G$-isogeny $A_0 \times P_G  \to A$, then showing that this induces an 
\'etale cover $A_0 \times (P_G/G) \to A/G$ whose Galois group is the kernel of the isogeny, $\D$ in the notation of \cite{AA18}.
In contrast, we first describe the quotient $E^g/\Sigma$ in explicit terms, then describe an \'etale cover of that. When $A$ is specialized to
$E^g$, $A_0$ becomes the connected component of the identity in the $\Sigma$-invariant subgroup of $E^{g+1}_0\cong E^g$;
thus, $A_0$ is $E^{|J|+s-1}$; $P_G$ is the subgroup $E_0^{I_1} \times \cdots \times E_0^{I_s}$; 
the Galois group $\D$ specializes to the Galois group in \cref{main.thm}; 
the fibration $A/G \to A_0/(A_0 \cap P_G)$ specializes to the fibration in \cref{thm1}; 
the \'etale cover in \cref{main.thm} is the quotient $A_0 \times (P_G/G) \to A/G$. We do not use the very nice idea  \cite{AA18} that the decomposition of the tangent space $T_0A$ at the identity as a direct sum of irreducible $G$-modules leads to an analogous decomposition  of $A$ (it is generated by $A_0$ and the product of the subgroups $A_1 \times \cdots \times A_r$);  
our $E^{I_\a}$'s and $\Sigma_{I_\a}$'s are special cases of the $A_i$'s (for $i>0$) and the $G_i$'s in \cite[Lem.~2.6]{AA18}.

\subsection{Acknowledgements}

A.C. was partially supported by NSF grant DMS-1801011.

R.K. was a JSPS Overseas Research Fellow, and supported by JSPS KAKENHI Grant Numbers JP16H06337, JP17K14164, and JP20K14288, Leading Initiative for Excellent Young Researchers, MEXT, Japan, and Osaka City University Advanced Mathematical Institute (MEXT Joint Usage/Research Center on Mathematics and Theoretical Physics JPMXP0619217849). R.K. would like to express his deep gratitude to Paul Smith for his hospitality as a host researcher during R.K.'s visit to the University of Washington.

\section{Results}  
\label{sect.results}

\subsection{What we already know}
The starting point for this paper is the description of  $E^g/\Sigma$ in \cite[Props.~1.3 and~4.21]{CKS2}. We state that result in 
\cref{thm1} below.

The partition $\{1,\ldots,g+1\}= J \sqcup I_1 \sqcup  \cdots \sqcup I_s$ corresponds to 
factorizations $E^{g+1}=  E^{J} \times E^{I_1}  \times \cdots \times E^{I_s}$ 
and $\Sigma=\Sigma_{I_1} \times \cdots \times \Sigma_{I_s}$. Both these are factorizations as a product of subgroups.
It follows that
\begin{equation}
\label{E^(g+1)/Sigma}
E^{g+1}/\Sigma \; \cong \;  E^{J} \times S^{I_1}E \times \cdots \times S^{I_s}E 
\end{equation}
where $S^{I_\a}E=E^{I_\a}/\Sigma_{I_\a}$ is the symmetric power of $E$ of dimension $i_\a$. We also write $S^{i_\a}E$ for $S^{I_\a}E$.  As is well-known, the summation function ${\sf sum}:S^rE \to E$ presents $S^rE$ as a $\PP^{r-1}$ bundle over $E$.  Thus, $E^{g+1}/\Sigma$ is a bundle over $E^{J} \times E^s$ with fibers isomorphic to $\PP^{i_1-1} \times \cdots \times \PP^{i_s-1}$.

\begin{proposition} 
\cite[Props.~1.3 and~4.21]{CKS2}
\label{thm1}
With the above notation, $E^g/\Sigma$ is a bundle over $E^{|J|+s-1}$ with fibers 
isomorphic to $\PP^{i_1-1} \times  \cdots \times \PP^{i_s-1}$. More precisely,
\begin{equation}
\label{E^g/Sigma}
E^g/\Sigma \; \cong \; 
\begin{cases} 
\big(S^{i_1}E \times  \cdots \times S^{i_s}E\big)_0 & \text{if $J=\varnothing$,}
\\
E^{|J|-1} \times S^{i_1}E \times  \cdots \times S^{i_s}E &\text{if  $J \ne \varnothing$,}
\end{cases}
\end{equation}
where $(S^{i_1}E \times  \cdots \times S^{i_s}E)_0$ denotes the subvariety of $S^{i_1}E \times  \cdots \times S^{i_s}E$ whose 
coordinates sum to $0$.
\end{proposition}

When $J\neq\varnothing$, the description of $E^{g}/\Sigma$ is obtained as the subvariety of $E^{J} \times S^{I_1}E \times \cdots \times S^{I_s}E$ whose coordinates sum to $0$, which is identified with $E^{|J|-1} \times S^{i_1}E \times  \cdots \times S^{i_s}E$ by dropping the last coordinate of $E^{J}$. This identification will be used in \cref{sect.lift}.

\subsection{An \'etale cover of $S^rE$}
Recall that $E[r]$ denotes the $r$-torsion subgroup of $E$.
We will now show that $S^rE$ is a quotient of $E \times \PP^{r-1}$ with respect to a free action of $E[r]$. 

We first recall the following result on free finite group actions. 

\begin{proposition}\label{pr.quot-etale}
  Let $G$ be a finite group acting freely on a variety $X$ over $\Bbbk$. Then, the quotient $\pi: X\to X/G$ is finite \'etale.
\end{proposition}
\begin{proof}
  According to \cite[\S 12, Thm.~1]{Mum08} the quotient is finite, faithfully flat and locally free. The same result also confirms that the canonical map
  \begin{equation*}
    G\times X\ni (g,x)\mapsto (gx,x) \in X\times_{X/G}X
  \end{equation*}
  is an isomorphism.

  It follows that the pullback
  \begin{equation*}
    \pi_1:X\times_{X/G}X\to X
  \end{equation*}
  of $\pi$ along itself is \'etale, being isomorphic to the second projection $G\times X\to X$. The fact that $\pi$ itself is \'etale now follows from \cite[Thm.~5.8]{len-gal} together with its aforementioned properties: finiteness, faithful flatness (hence surjectivity) and local freeness.
\end{proof}

\begin{lemma}
\label{lem0}
The morphism $\Psi:E_0^r \times E \to E^r$ given by the formula
\begin{equation}
\label{quotient.map.0}
\Psi (x_1,\ldots,x_r,z)  \, := \,  (x_1+z,\ldots,x_r+z)
 \end{equation}
 is \'etale with Galois group $E[r]$ if and only if $r$ is coprime to the characteristic of the ground field $\Bbbk$. The fibers of $\Psi$ are the orbits with respect to the action of $E[r]$ on $E_0^r \times E$ given by the formula
\begin{equation}
\label{action0}
\xi \, \triangleright \, (x_1,\ldots,x_r,z) \; = \; (x_1+\xi,\ldots,x_r+\xi,z-\xi).
\end{equation}
\end{lemma}
\begin{proof}
  Denoting the Lie algebra of $E$ by $V$, the differential $d\Psi$ at the origin is the map
  \begin{equation*}
    V_0^r\times V\ni (v_1,\ldots,v_r,v)\mapsto (v_1+v,\ldots,v_r+v)\in V^r
  \end{equation*}
  where $V_0^r$ denotes the set of $r$-tuples in $V$ summing to $0$. This map is an isomorphism precisely when $\mathrm{char}\;\Bbbk$
  does not divide $r$, so that condition is indeed necessary.

As for sufficiency, identify $E[r]$ with the subgroup 
$$
T \; :=\; \{(\xi,\ldots,\xi,-\xi) \; | \; \xi \in E[r]\}
$$
of $E_0^{r} \times E$ via the group isomorphism $\xi \longleftrightarrow (\xi,\ldots,\xi,-\xi)$ and observe that the action of $E[r]$ in \cref{action0} becomes identified with the translation action of $T$. $\Psi$ now induces a bijection
\begin{equation*}
  (E_0^r\times E)/T\to E^r
\end{equation*}
of abelian varieties whose differential is injective on tangent spaces (because it is so at the origin). The fact that this map is an isomorphism now follows from \cite[Cor.~14.10]{H92}, and the \'etaleness claim follows from \Cref{pr.quot-etale}.
\end{proof}

The point in $S^rE$ that is the image of $(x_1,\ldots,x_r) \in E^r$ will be denoted by
$$
(\!(x_1,\ldots,x_r)\!).
$$
In the next lemma, we identify $\PP^{r-1}$ with the subvariety $S_0^rE \subseteq S^rE$ 
consisting of those points whose coordinates sum to zero. We define the action of a point $\xi \in E[r]$ on $\PP^{r-1} \times E=S_0^rE \times E$ by 
\begin{equation}
\label{Er.action}
\xi \triangleright  \big((\!(x_1,\ldots,x_r)\!),z\big)   \; :=  \; \big((\!(x_1+\xi,\ldots,x_r+\xi)\!),z-\xi\big).
\end{equation} 
This is a free action of $E[r]$ on $\PP^{r-1} \times E=S_0^rE \times E$ simply because translation by $-\xi$ is a free action.

\begin{lemma}
\label{lem1}
The map $\Phi:\PP^{r-1} \times E \to S^rE$ given by the formula
\begin{equation}
\label{the.map.Phi}
\Phi \big((\!(x_1,\ldots,x_r)\!),z\big) := (\!(x_1+z,\ldots,x_r+z)\!)
\end{equation}
is the \'etale quotient with respect to the action of $E[r]$ in \cref{Er.action} if and only if $\gcd\{\mathrm{char}(\Bbbk),r\}=1$
\end{lemma}
\begin{proof}
  Let $\Psi:E_0^r \times E \to E^r$ be the morphism in \cref{quotient.map.0}.  Let $\b:E^r \to S^rE$ be the natural map, i.e., the quotient with respect to the natural action of $\Sigma_r$, and let $\a:E_0^r \times E \to S_0^r E \times E$ be the morphism $\a(\sfx,z)=(\b(\sfx),z)$.

  The diagram
$$
\xymatrix{
E^r_0 \times E \ar[d]_\Psi   \ar[rr]^\a    &&  S^r_0E \times E \ar[d]^\Phi 
\\
  E^r  \ar[rr]_\b   &&   S^rE
}
$$
is commutative, and $\a$ is equivariant for the actions of $E[r]$ on $E_0^r \times E$ and $S_0^r E \times E$ in \cref{action0} and \cref{Er.action}.

We observed in the proof of \Cref{lem0} that $\Psi$ does not induce an isomorphism 
on tangent spaces when $\mathrm{char}(\Bbbk)$ divides $r$. On the other hand, the symmetric group actions are free on open dense 
subschemes of the left-hand corners of the diagram, so $\Phi$ is not \'etale when $\mathrm{char}(\Bbbk)$ divides $r$.

Conversely, assume that $\mathrm{char}(\Bbbk)$ does not divide $r$. 

The action of $\Sigma_r$ on $E^r_0 \times E$ permuting the first $r$ coordinates and leaving the last coordinate unchanged commutes with the action of $E[r]$, so there is an action of $\Sigma_r \times E[r]$ on $E^r_0 \times E$.  Since $\Psi$ is the quotient with respect to the action of $E[r]$ and $\b$ the quotient with respect to the action of $\Sigma_r$, the morphism $\b\Psi:E_0^r \times E \to S^rE$ is the quotient with respect to the action of $\Sigma_r \times E[r]$.

In other words, $\Phi\a$ is the quotient with respect to the action of $\Sigma_r \times E[r]$.  But $S_0^rE$ is the quotient of $E^r_0$ with respect to the action of $\Sigma_r$, so $\a$ is also the quotient with respect to the action of $\Sigma_r$ on $E_0^r \times E$. It follows that $\Phi$ must be the quotient with respect to the action of $E[r]$ on $S_0^rE \times E$; one proves this by showing that $\Phi$ has the appropriate universal property for morphisms $S_0^rE \times E \to X$ that are constant on $E[r]$-orbits. 
Since the action of $E[r]$ on $S_0^rE \times E$ is free, $\Phi$ is \'etale by \Cref{pr.quot-etale}.
\end{proof}

The next two results extend \cref{lem1} to $E^g/\Sigma$ using its description in \cref{E^g/Sigma}.

\subsection{The case $J\neq\varnothing$}
\label{sect.easy.case}
 First, the easy case.

\begin{proposition}
\label{prop.easy.case}
Assume $\mathrm{char}(\Bbbk)$ does not divide the product $i_1\cdots i_s$.  If $J \ne \varnothing$, 
there is an \'etale cover
$$
\Psi:E^{|J|+s-1} \times \PP^{i_1-1} \times  \cdots \times \PP^{i_s-1} 
\; \longrightarrow \; E^g/\Sigma
$$
with Galois group $E[i_1] \times \cdots \times E[i_s]$.
\end{proposition}
\begin{proof}
In this proof we identify each $\PP^{i_\a-1}$  in the statement of the proposition with 
$S^{i_\a}_0E$, the subvariety of $S^{i_\a}E=S^{I_\a}E$ consisting of those points whose coordinates sum to $0$.

Since $J \ne \varnothing$, $E^g/\Sigma \cong E^{|J|-1} \times S^{i_1}E \times  \cdots \times S^{i_s}E$. 
By combining the maps in \cref{the.map.Phi} for each $S^{i_\a}E$, it follows from \cref{lem1} that there is an \'etale cover
\begin{equation}
\label{etale.cover}
\Psi: E^{|J|-1} \times (S_0^{i_1}E \times E) \times  \cdots \times (S_0^{i_s}E \times E) \; \longrightarrow \; E^g/\Sigma
\end{equation}
with Galois group $\Gamma:=E[i_1] \times \cdots \times E[i_s]$. 

To make  the action of $\Gamma$ explicit, we 
write the domain in \cref{etale.cover} as 
$$
E^{|J|-1} \times E^s \times S_0^{i_1}E \times  \cdots \times S_0^{i_s}E
$$
and points in it as 
\begin{equation}
\label{a.point}
(\sfy,z_1,\ldots,z_s,\sfx_1,\ldots,\sfx_s)
\end{equation}
where $\sfy \in E^{|J|-1}$, each $z_\a$ is in $E$,  and each $\sfx_\a$ is in $S_0^{i_\a}E$. 
 If $\xi=(\xi_1,\ldots,\xi_s) \in \Gamma$, then 
$E^g/\Sigma$ is the quotient of $E^{|J|-1} \times E^s \times S_0^{i_1}E \times  \cdots \times S_0^{i_s}E$ with respect to the action
$$
\xi \triangleright (\sfy,z_1,\ldots,z_s,\sfx_1,\ldots,\sfx_s) \; =\; (\sfy,z_1-\xi_1,\ldots,z_s-\xi_s,\xi_1\triangleright \sfx_1,\ldots,\xi_s\triangleright \sfx_s)
$$
where 
$$
\xi_\a \triangleright \sfx_\a  \;=\;  \xi_\a \triangleright  (\!(x_{\a, 1},x_{\a, 2},\ldots,x_{\a,i_\a})\!)  \;=\; 
    (\!(x_{\a, 1}+\xi_\a,x_{\a, 2}+\xi_\a,\ldots, x_{\a,i_\a}+\xi_\a  )\!)
$$
for $\a=1,\ldots,s$. 
\end{proof}

\subsection{The case $J= \varnothing$}
\label{sect.hard.case}
Assume $J=\varnothing$.

In this case,
$$
E^g/\Sigma \; \cong \; \big(S^{i_1}E \times  \cdots \times S^{i_s}E\big)_0.
$$
It follows from \cref{lem1} that there is an \'etale cover 
\begin{equation}
\label{etale.cover2}
\Psi: E^s \times S_0^{i_1}E  \times  \cdots \times S_0^{i_s}E  \; \longrightarrow \;  S^{i_1}E \times  \cdots \times S^{i_s}E
\end{equation}
with Galois group $E[i_1] \times \cdots  \times E[i_s]$ given by the formula
\begin{align*}
\Psi(z_1,\ldots,z_s,\sfx_1,\ldots,\sfx_s)  \; :=\; & (\sfx_1 + z_1, \ldots , \sfx_s + z_s)
\\
 \;=\; & \big( (\!(x_{1,1}+z_1,\ldots,x_{1,i_1}+z_1)\!), \ldots, (\!(x_{s,1}+z_s,\ldots,x_{s,i_s}+z_s)\!)\big).
\end{align*}
The point $\Psi(z_1,\ldots,z_s,\sfx_1,\ldots,\sfx_s)$ belongs to $\big(S^{i_1}E \times  \cdots \times S^{i_s}E\big)_0$ if and 
only if $i_1 z_1+\cdots + i_s z_s =0$.

The map $\Theta:E^s \to E$ defined by 
\begin{equation}
\label{defn.Theta}
\Theta (z_1,\ldots,z_s) \; :=\; i_1 z_1+\cdots + i_s z_s
\end{equation}
is a group homomorphism, and the \'etale cover $\Psi$ restricts to an \'etale cover
\begin{equation}
\label{etale.cover3}
\Psi: \ker(\Theta) \times S_0^{i_1}E  \times  \cdots \times S_0^{i_s}E  \; \longrightarrow \;  \big(S^{i_1}E \times  \cdots \times S^{i_s}E\big)_0 \, \cong \, E^g/\Sigma.
\end{equation}

\begin{lemma}\label{le.kertheta}
The kernel of the  homomorphism $\Theta:E^s \to E$ defined by the formula
\begin{equation}
\label{defn.Theta2}
\Theta (z_1,\ldots,z_s) \; :=\; i_1 z_1+\cdots + i_s z_s
\end{equation}
is isomorphic to $E[d] \times E^{s-1}$ where $d=\gcd\{i_1,\ldots,i_s\}$.
\end{lemma}
\begin{proof}
  Multiplication by some invertible matrix $M\in GL_s(\ZZ)$ implements the transformation
  \begin{equation*}
    (i_1,i_2,\ldots,i_s)\mapsto (d,0,\ldots,0).
  \end{equation*}
  Composing $\Theta$ with the automorphism of $E^s$ corresponding to the transpose $M^{t}$ of $M$ will produce the morphism
  \begin{equation*}
    \Theta\circ M^{t}:(z_1,z_2,\ldots,z_s)\mapsto (dz_1,0,\ldots,0),
  \end{equation*}
  whose kernel is as described in the statement.
\end{proof}

\begin{proposition}\label{pr.hard-case}
  Assume $\mathrm{char}(\Bbbk)$ does not divide the product $i_1\cdots i_s$.  If $J = \varnothing$, there is an \'etale cover
$$
\Psi:E^{s-1} \times \PP^{i_1-1} \times  \cdots \times \PP^{i_s-1} 
\; \longrightarrow \; E^g/\Sigma
$$
with Galois group $E[i_1] \hdot \cdots \hdot E[i_s]$.
\end{proposition}
\begin{proof}
Let $\Theta$ be the map in \Cref{defn.Theta2} and let $\ker(\Theta)^0$ be the connected component of the identity   in the algebraic group $\ker(\Theta)$.

We have already seen that \Cref{etale.cover3} provides an \'etale cover $\Psi$ with Galois group $\Gamma:=E[i_1]\times\cdots \times E[i_s]$. By \Cref{le.kertheta}, the domain of that cover is a disjoint union of $d^2$ copies of $E^{s-1} \times \PP^{i_1-1} \times  \cdots \times \PP^{i_s-1}$ and, since the codomain is connected, the restriction of $\Psi$ to any one of those copies/sheets is an \'etale cover, and 
the stabilizer of that sheet in $\Gamma$ is the Galois group.
In particular,  the sheet $\ker(\Theta)^0 \times S_0^{i_1}E  \times  \cdots \times S_0^{i_s}E$ provides an \'etale cover; 
it is immediate that the stabilizer group  in $\Gamma$  of this sheet is the right-hand side of \Cref{eq:1}. This concludes the proof.
\end{proof}

  \cref{main.thm} is obtained by combining the results in \cref{prop.easy.case,pr.hard-case}.

\subsection{Lifting automorphisms of $E^g/\Sigma$ to its \'etale cover}
\label{sect.lift}

When $E^g/\Sigma$ appears in the context of the elliptic algebras mentioned in \cref{origin1} it is endowed with a distinguished
automorphism $\s$ that is induced by an automorphism of $E^g$ that is translation by a point $\sft$ that is fixed by $\Sigma$. We
will now show that this induces an automorphism of the \'etale cover in   \cref{main.thm}. The elliptic algebras are defined over $\CC$
but the result is true in greater generality.

\begin{proposition}
Assume ${\rm char}(\Bbbk)$ does not divide the product $i_1 \cdots i_s$.
Let  $\s:E^g \to E^g$ be a translation automorphism $\s(\sfy)=\sfy+\sft$. 
If $\sft$ is fixed by $\Sigma$, then $\s$ descends to an automorphism of $E^g/\Sigma$ that, in turn, lifts to an automorphism
of the \'etale cover described in   \cref{main.thm} and \cref{sect.easy.case,sect.hard.case}.
\end{proposition}
\begin{proof}
Suppose  $\sft$ is $\Sigma$-invariant.

Recall that $E^{g+1}=E^J \times E^{I_1} \times \cdots \times E^{I_s}$ is a factorization as a product
of subgroups. When we identify $E^g$ with the subgroup $E_0^{g+1}$ of $E^{g+1}$ we can, and will, write $\sft$ as
$$
\sft \;=\; (\sft_J,\sft_1,\ldots,\sft_s)
$$
where $\sft_J \in E^J$ and $\sft_\a \in E^{I_\a}$ for each $\a$. If $J=\varnothing$, then $\sft_J=0\in \{0\}=E^J$. 
As pointed out in \cite[Prop.~2.10]{CKS2}, $\s$ descends to an automorphism of $E^g/\Sigma$ because $\sft$ is $\Sigma$-invariant.

Since $E_0^{g+1}\subset E^g$ is stable under the translation by $\sft$, we have
\begin{equation}\label{eq.sum.zero}
	\sfsum(\sft_{J})+\sfsum(\sft_{1})+\cdots+\sfsum(\sft_{s})=0.
\end{equation}
Moreover, the fact that $\sft$ is fixed by each subgroup $\Sigma_{I_\a}$ of $\Sigma$ implies that $\sft_\a$ is a diagonal element, 
$\sft_\a=(t_\a,\ldots,t_\a)$ for some $t_\a \in E$. Hence
\begin{equation*}
	\sfsum(\sft_{J})+i_{1}t_{1}+\cdots+i_{s}t_{s}=0.
\end{equation*}
The automorphism 
$$
\big((\!(x_1,\ldots,x_{i_\a})\!),z\big) \mapsto  \big((\!(x_1,\ldots,x_{i_\a})\!),z+t_\a\big)
$$
of  the \'etale cover $\PP^{i_\a-1} \times E$ of  $S^{i_\a}E=S^{I_\a}E$ (in \cref{lem1}) provides a lift of that
automorphism 
$$
(\!(x_1,\ldots,x_{i_\a})\!) \mapsto (\!(x_1+t_\a,\ldots,x_{i_\a}+t_\a)\!),
$$
of $S^{I_\a}E$ that is induced by the translation automorphism $\sfx \mapsto \sfx+\sft_\a$ of $E^{I_\a}$.

Suppose $J \ne \varnothing$. Write $\sft_{J}=(\sft',t)$ where $\sft'\in E^{|J|-1}$ and $t\in E$. Using the notation in the proof of \cref{prop.easy.case}, the automorphism
$$
(\sfy,z_1,\ldots,z_s,\sfx_1,\ldots,\sfx_s) \mapsto (\sfy+\sft',z_1+t_1,\ldots,z_s+t_s,\sfx_1,\ldots,\sfx_s)
$$
of the cover $E^{|J|-1} \times E^s \times S_0^{i_1}E \times  \cdots \times S_0^{i_s}E$ provides a lift of $\s:E^g/\Sigma \to E^g/\Sigma$. Indeed, if we regard $E^{g}/\Sigma$ as the subvariety of $E^{J} \times S^{I_1}E \times \cdots \times S^{I_s}E$, the translation on the last coordinate of $E^{J}$ should be the addition of $t$ because of \cref{eq.sum.zero}. Hence the proposition is true when  $J \ne \varnothing$. 

Suppose $J = \varnothing$. Thus, $\sft \;=\; (\sft_1,\ldots,\sft_s)$.
We will use the notation in \cref{sect.hard.case}. 
Note that the component $\ker(\Theta)^{0}$ is the kernel of $\tfrac{1}{d}\Theta\colon E^{s}\to E$. Take $q\in E$ such that
\begin{equation*}
	\frac{i_{1}}{d}q\;=\;\frac{i_{1}}{d}t_{1}+\cdots+\frac{i_{s}}{d}t_{s}.
\end{equation*}
Equation \cref{eq.sum.zero} implies that
\begin{equation*}
	i_{1}q\;=\;i_{1}t_{1}+\cdots+i_{s}t_{s}\;=\;0.
\end{equation*}
Thus $q\in E[i_{1}]$.
It follows that $\ker(\Theta)^{0} \times S_0^{i_1}E  \times  \cdots \times S_0^{i_s}E$ is stable under  
the ``translation automorphism''
$$
(z_1,\ldots,z_s,\sfx_1,\ldots,\sfx_s) \mapsto (z_1+t_1-q,z_{2}+t_{2},\ldots,z_s+t_s,\Phi(\sfx_1,q),\sfx_{2},\ldots,\sfx_s)
$$
of $E^s \times  S^{i_1}E  \times  \cdots \times S^{i_s}E$, where
\begin{equation*}
	\Phi \big((\!(x_1,\ldots,x_{i_{1}})\!),q\big) := (\!(x_1+q,\ldots,x_{i_{1}}+q)\!).
\end{equation*}
The restriction of that ``translation automorphism''
to $\ker(\Theta)^{0} \times S_0^{i_1}E  \times  \cdots \times S_0^{i_s}E$ provides the desired lift of $\s$ to the \'etale cover.
\end{proof}

\bibliographystyle{customamsalpha}

 \def\cprime{$'$}
\providecommand{\bysame}{\leavevmode\hbox to3em{\hrulefill}\thinspace}
\providecommand{\MR}{\relax\ifhmode\unskip\space\fi MR }
\providecommand{\MRhref}[2]{%
  \href{http://www.ams.org/mathscinet-getitem?mr=#1}{#2}
}
\providecommand{\href}[2]{#2}

\end{document}